\documentclass[12pt,UKenglish]{amsart}
\usepackage[margin=25 mm,bottom=35mm,footskip=15mm,top=35mm]{geometry}
\usepackage{enumerate,url,amssymb,amsthm,amsmath,xstring,enumitem,hyperref,subcaption,floatrow,setspace}
\usepackage[UKenglish]{babel}
\usepackage[numbers]{natbib}
\usepackage{xcolor}
\usepackage{graphicx}
\usepackage[T1]{fontenc}

\usepackage{aliascnt}
\usepackage[capitalise,noabbrev]{cleveref}
\crefname{appsec}{Appendix}{Appendices}
\crefformat{equation}{#2(#1)#3}

\theoremstyle{plain}

\newtheorem{theorem}{Theorem}[section]

\newaliascnt{proposition}{theorem}

\aliascntresetthe{proposition}

\newaliascnt{lemma}{theorem}
\newtheorem{lemma}[lemma]{Lemma}
\aliascntresetthe{lemma}

\newaliascnt{claim}{theorem}

\aliascntresetthe{claim}

\newaliascnt{corollary}{theorem}

\aliascntresetthe{corollary}

\newaliascnt{conjecture}{theorem}

\aliascntresetthe{conjecture}

\newaliascnt{observation}{theorem}

\aliascntresetthe{observation}

\newtheorem*{question*}{Question}

\theoremstyle{definition}
\newaliascnt{definition}{theorem}

\aliascntresetthe{definition}

\newaliascnt{question}{theorem}

\aliascntresetthe{question}

\newaliascnt{example}{theorem}

\aliascntresetthe{example}

\newaliascnt{problem}{theorem}
\newtheorem{problem}[problem]{Problem}
\aliascntresetthe{problem}

\theoremstyle{remark}

\crefname{theorem}{Theorem}{Theorems}
\crefname{proposition}{Proposition}{Propositions}
\crefname{lemma}{Lemma}{Lemmas}
\crefname{claim}{Claim}{Claims}
\crefname{corollary}{Corollary}{Corollaries}
\crefname{conjecture}{Conjecture}{Conjectures}
\crefname{observation}{Observation}{Observations}
\crefname{definition}{Definition}{Definitions}
\crefname{question}{Question}{Questions}
\crefname{example}{Example}{Examples}
\crefname{problem}{Problem}{Problems}

\Crefname{theorem}{Theorem}{Theorems}
\Crefname{proposition}{Proposition}{Propositions}
\Crefname{lemma}{Lemma}{Lemmas}
\Crefname{claim}{Claim}{Claims}
\Crefname{corollary}{Corollary}{Corollaries}
\Crefname{conjecture}{Conjecture}{Conjectures}
\Crefname{observation}{Observation}{Observations}
\Crefname{definition}{Definition}{Definitions}
\Crefname{question}{Question}{Questions}
\Crefname{example}{Example}{Examples}
\Crefname{problem}{Problem}{Problems}

\newcommand{\xqed}[1]{%
	\leavevmode\unskip\penalty9999 \hbox{}\nobreak\hfill
	\quad\hbox{\ensuremath{#1}}}
\newcommand{\Endofdef}{\xqed{\lozenge}}

\title[Supersaturation of induced even cycles in locally sparse graphs]{Supersaturation of induced even cycles in locally sparse graphs}

\author[D\v{z}avoronok]{Adam D\v{z}avoronok}
\address{Department of Applied Mathematics, Charles University, Faculty of Mathematics and Physics, Malostransk\'e n\'am.~25, 118~00 Praha~1, Czech Republic}
\email{adam.dzavoronok@mff.cuni.cz}

\author[Gabsdil]{Ole Gabsdil}
\address{Department of Mathematics, EPFL, Switzerland}
\email{ole.gabsdil@epfl.ch}

\author[Mylet]{Alexander Mylet}
\address{Department of Computer Science, University of Oxford, Parks Road, Oxford, OX1 3QD, United Kingdom}
\email{alexander.mylet@sjc.ox.ac.uk}
\author[Popa]{Maria-Cristina Popa}
\address{Department of Mathematics, École Polytechnique, Institut Polytechnique de Paris, 91120 Palaiseau, France}

\email{maria-cristina.popa@polytechnique.edu}
\author[Shao]{Yinghan Andie Shao}
\address{Department of Mathematics, EPFL, Switzerland}
\email{yinghan.shao@epfl.ch}

\begin{document}
\begin{abstract}
A graph $\Gamma$ is \emph{$(c,t)$-sparse} for $c > 0$ and $t \ge 1$ if for every pair of vertex subsets $A, B \subseteq V(\Gamma)$ with $|A|, |B| \ge t$, the number of edges $e(A,B)$ between them satisfies
$ e(A,B) \le (1 - c)|A||B|$.
In this paper, we prove that for every integer $\ell\ge2$, there are $\varepsilon > 0, C, C' > 0$ such that if an $n$-vertex graph $\Gamma$ is $(1-\varepsilon,t)$-sparse for some $t$, and has at least $Ct^{1-1/\ell}n^{1+1/\ell}$ edges, then $\Gamma$ contains at least $C'n^2t^{2\ell-2}$ induced copies of $C_{2\ell}$.
This partially resolves a problem of Ding, Gao, Liu, Luan, and Sun.

\end{abstract}

	\maketitle

\section{Introduction}

For a graph $H$, the Tur\'an number $\operatorname{ex}(n,H)$ is the maximum number of edges in an $n$-vertex graph containing no copy of $H$.
One of the central topics studied in extremal graph theory is Tur\'an-type problems, of which induced variants have had some focus in recent years.
In the classical theory, the asymptotics are determined by the chromatic number of $H$. The Erd\H os--Stone--Simonovits theorem \cite{ErdosP1946Otso, ErdosP1966Alti} states that
\(
    \operatorname{ex}(n,H) = \left(1 - \frac{1}{\chi(H)-1} + o(1)\right)\binom{n}{2}
\)
for every graph $H$ with at least one edge, so that $\operatorname{ex}(n,H) = o(n^2)$ precisely when $H$ is bipartite.
The bipartite case is much less understood. In particular, for the even cycle $C_{2\ell}$  of length $2\ell$,
the  theorem of Bondy and Simonovits provides an upper bound $\operatorname{ex}(n,C_{2\ell})=O_{\ell} (n^{1+1/\ell})$~\cite{bondy-simonovits}, but matching lower bounds are known only in cases when $\ell=2,3$ or $5$ (see \cite{Wenger1991}).
Determining the exact asymptotic behaviour of $\operatorname{ex}(n,C_{2\ell})$ remains a long-standing open problem.

For induced copies no such statement is possible.
When the forbidden induced graph $H$ is not a clique, then the complete graph $K_n$ does not contain any induced copies of $H$, and so the question only becomes meaningful with an additional sparsity restriction.
Loh, Tait, Timmons, and Zhou \cite{LOHPO-SHEN2018ITN} introduced the question of determining the maximum number of edges in an $n$-vertex graph forbidding both an induced copy of $H$ and a copy of some graph $F$ (not necessarily induced).
The case of $F = K_{t,t}$ has been well-studied (see, among others, \cite{AxenovichMaria2025ITpi, BourneufRomain2024Opd, HunterZach2025Ktfh, BonamyMarthe2022DoPa,KuhnDaniela2004ISIK}). 
In particular, Hunter, Milojevi\'{c},
Sudakov, and Tomon \cite{HunterZach2025Ktfh} established that in a $K_{t,t}$-free graph with at least $C_{\ell,t}n^{1+1/\ell}$ edges, there is an induced copy of $C_{2\ell}$.

Fox, Nenadov, and Pham \cite{fox-nenadov-pham} initiated the systematic study of induced Tur\'an problems in $(c, t)$-sparse graphs.
A graph $G$ is called \emph{$(c,t)$-sparse} for $c > 0$ and $t \ge 1$ if for every pair of vertex subsets $A, B \subseteq V(G)$ with $|A|, |B| \ge t$, the number of edges $e(A,B)$ between them satisfies
\(
e(A,B) \le (1 - c)|A||B|
\).
$A$ and $B$ do not have to be disjoint, and we double count edges within $A \cap B$.
This notion generalises $K_{t,t}$-free graphs, since by averaging, $G$ is $K_{t,t}$-free if and only if it is $\left(\frac{1}{t^2}, t\right)$-sparse.
Fox, Nenadov, and Pham conjectured that for every $c > 0$ and integer $\ell$ there is a $C > 1$ such that if $\Gamma$ is a $(c,t)$-sparse graph for some $t$ with at least $Ct^{1-1/\ell}n^{1+1/\ell}$ edges, then $\Gamma$ contains an induced copy of $C_{2\ell}$.
This was proved by Ding et al.~\cite{ding2024induced}.

Another central question in the study of Tur\'an-type problems is to establish the corresponding supersaturation results: once the number of edges exceeds the extremal threshold by a fixed constant factor, how many copies of the forbidden graph must necessarily occur?
For the even cycle $C_{2\ell}$, in the classical case, this is a result of Simonovits \cite{ErdosP1984Csga}: for every integer $\ell \ge 2$, there are constants $C, c > 0$ such that every $n$-vertex graph $G$ with $e(G) \ge Cn^{1+1/\ell}$ contains at least  $c\left(\frac{e(G)}{n}\right)^{2\ell}$ copies of $C_{2\ell}$.
Moreover, the random graph of the same edge density proves to be a tight example up to the multiplicative constant.
More recently, Jiang and Yepremyan \cite{jiangYepremyan2020} reproved this in the more general setting of even linear cycles in linear $r$-graphs.
In the induced $(c,t)$-sparse setting, Fox, Nenadov, and Pham established supersaturation results for various classes of bipartite graphs~\cite{fox-nenadov-pham}.

Ding et al.\@ \cite{ding2024induced} posed the problem of proving, if true, the corresponding supersaturation result for induced copies of $C_{2\ell}$ under $(c,t)$-sparseness assumptions:
\begin{problem}\label{problem:the problem}
For every $c > 0$ and integer $\ell$, do there exist $C$ and $C'>0$ such that if an $n$-vertex graph $G$ is $(c,t)$-sparse for some $t$, and has at least $Ct^{1-1/\ell}n^{1+1/\ell}$ edges, then $G$ contains at least $C't^{2\ell-2}n^2$ induced copies of $C_{2\ell}$.
\end{problem}
Relatedly, Dong, Gao, Li, and Liu~\cite{DongGaoLiLiu2025} proved a similar tight supersaturation result for induced even cycles in sufficiently dense almost-regular \(K_{s,s}\)-free graphs with exponential dependence in $s$ in the number of necessary edges.
Furthermore, the $C_4$-case was fully resolved in $(c,t)$ sparse setting by Fox, Nenadov and Pham~\cite{fox-nenadov-pham}.
In this paper, we show that \cref{problem:the problem} is true for sufficiently sparse graphs:

\begin{theorem}\label{theorem:main theorem}
    For every integer $\ell\ge2$, there exist $\varepsilon > 0$, $C$ and $C' > 0$, all depending on $\ell$, such that if an $n$-vertex graph $\Gamma$ is $(1-\varepsilon, t)$-sparse for some $t$, and has at least $Ct^{1-1/\ell}n^{1+1/\ell}$ edges, then $\Gamma$ contains at least $C'd^{2\ell} 
   $ 
    induced copies of $C_{2\ell}$, where $d$
    is the average degree of $\Gamma$.
\end{theorem}

We remark that from our proof we recover that $\varepsilon$ in the statement may be chosen at most of order $\ell^{-2}$. 
Our proof also provides an alternative proof of the original induced Tur\'an problem for $(c,t)$-sparse graphs, which follows directly from the following lemma: 
\begin{lemma}[Lemma 3.1, \cite{ding2024induced}]
Let \(H\) be a bipartite graph. For any \(c,\varepsilon>0\), there exists
\(\beta_0=\beta(c,H,\varepsilon)
\) such that for any \(\beta\geq  \beta_0\) and every \((c,t)\)-sparse graph \(G\), if \(G\) does not contain an induced copy of \(H\), then it is also \((1-\varepsilon,\beta t)\)-sparse.
\end{lemma}
  If there is no induced copy of $C_{2\ell}$, then this lemma applies to bootstrap from $(c, t)$-sparseness to $(1-\varepsilon, \beta t)$-sparseness and our result shows there will be at least one induced $C_{2\ell}$ if there are at least $C(\beta t)^{1-1/\ell}n^{1+1/\ell}$ edges.

\subsection*{Proof overview.} 
Our approach uses a known regularisation result from \cite{jiangYepremyan2020} (see Lemma~\ref{lemma:regularization_citation}), splitting the graph $\Gamma$ into multiple edge-disjoint (not necessarily induced) subgraphs $G_i$ with only a constant factor between minimum and maximum degree. We lower-bound the number of copies of $C_{2\ell}$ which occur in each subgraph $G_i \subseteq \Gamma$ that are induced in $\Gamma$. 
For this purpose, we use an algebraic approach of counting the number of homomorphic copies of $C_{2\ell}$ occurring in each $G_i$.
A result of Janzer \cite{JanzerOliver2023RTno} (see Lemma~\ref{lemma:many injective homs0}) allows us to bound the number of non-injective homomorphisms.
The setup assembling these previous results is explained in \Cref{section:technical}.

The homomorphism counting technique was already used in the $K_{s,s}$-free setting by Hunter, Milojevi\'c, Sudakov, and Tomon ~\cite{HunterZach2025Ktfh}, and subsequently by Dong, Gao, Li, and Liu ~\cite{DongGaoLiLiu2025} to obtain the supersaturation result. In the $(c,t)$-sparse setting, our main new ingredient is a quantitative spread-outness estimate for the distributions of vertices along the walks forming a homomorphic cycle. Under the stronger assumption $c=(1-\varepsilon)$ with sufficiently small $\varepsilon$, this estimate yields the conjectured polynomial dependence on $t$.

More precisely, to control a fixed prospective chord, we condition on the images $x$ and $y$ of two opposite positions of the cycle. The two halves of a uniformly random homomorphism then become independent uniformly chosen walks of length $\ell$ from $x$ to $y$. At each internal position, we call a vertex heavy if it occurs in at least a $1/t$ proportion of these walks. The key quantitative estimate (\cref{lemma:rho_r small}) bounds the average mass of the heavy vertices. 
After discarding all heavy vertices, no vertex can appear in a fixed position in more than $1/t$ of the $xy$-walks. This spread of the walks allows us to use the $(1-\varepsilon, t)$-sparseness of $\Gamma$, expressed through the bilinear form associated with its adjacency matrix (\cref{lemma:bilinear form bounded,lemma:small probability of chords}) to bound the probability that the two positions are adjacent by $\varepsilon$.
A union bound over all nonconsecutive pairs of positions shows that a positive proportion of the homomorphisms in each $G_i$ are induced cycles in $\Gamma$. Summing over the edge-disjoint subgraphs $G_i$ completes the proof.
\subsection*{Notation} Let $G$ be a graph. We write $v(G)$ and $e(G)$ for   the number of vertices and edges of $G$ respectively. By $\delta(G)$, $d(G)$, and $\Delta(G)$, we denote the minimum, average and maximum degree of $G$ respectively.
A graph is $K$-almost-regular if $\Delta (G) \leq K\delta (G)$.
A \emph{homomorphism} from a graph $H$ to a graph $G$ is a mapping $f: V(H) \to V(G)$ that preserves adjacency, which means that if $(u, v)$ is an edge in $H$, then $(f(u), f(v))$ is an edge in $G$.
We write $\operatorname{Hom}(H,G)$ for the collection of such homomorphisms, and $\hom(H, G) = \lvert\operatorname{Hom}(H,G)\rvert$ for the number of such homomorphisms from $H$ to $G$.
A homomorphism $f: H \to G$ is \emph{injective} if it is injective in the usual way.

\section{Controlling Chords}\label{section:inducedness}
In this section, we prove the estimates on heavy vertices and prospective chords described in the proof overview. We now introduce some notation relevant for this section.
Note that much of this notation leaves the graph $G$ implicit (and when dealing with two graphs $G \subseteq \Gamma$, the notation will always refer to $G$).
We also leave the length $\ell$ implicit, where we are counting induced copies of $C_{2\ell}$.
For an integer $j$, we denote by $w_{xv}^{j}$ the number of walks of length $j$ starting at $x$ and ending at $v$ in $G$.
By $$p_{xy}^r(v)=\frac{w^r_{xv}w^{\ell-r}_{vy}}{w^\ell_{xy}},$$ we denote the probability of the vertex $v$ being in the $r$-th position on a randomly chosen walk of length $\ell$ between $x$ and $y$ if $w^\ell_{xy}\neq0$, and $p_{xy}^r(v) = 0$ otherwise.
As already said, a vertex $v$ occurring at position $r$ on many $xy$-walks is \emph{heavy at position $r$} if $p^r_{xy}(v) \ge 1/t$, and we denote by $$\alpha_{xy}^r=\sum_{p_{xy}^r(v)\ge\frac{1}{t}}p_{xy}^r(v)$$ the probability mass of the heavy vertices.
Lastly, we denote by $$\rho_r=\sum _{(x,y)}\frac{(w_{xy}^{\ell})^2}{\hom(C_{2\ell}, G)}\alpha_{xy}^r,$$
the averaged probability mass over all pairs $x,y$.

We begin the proof by recalling two standard estimates for the number of homomorphisms of an even cycle.

\begin{lemma}[e.g.\@ Theorem 1.1 and Corollary 1.3 \cite{StanleyRichardP2013ACWT}]\label{lemma:H_2l eigenvalues and paths}
Let $G$ be a graph on $n$ vertices with adjacency eigenvalues $\lambda_1,\dots,\lambda_n$. Then, for every positive integer $\ell\ge2$ we have
$$\hom(C_{2\ell},G)
 =\sum_{i=1}^n\lambda_i^{2\ell}
 =\sum_{x,y\in V(G)}(w_{xy}^{\ell})^2.$$
\end{lemma}
\begin{lemma}[\cite{SIDORENKOA.F.1992Iffg}] \label{lemma:many homs}
    If $\ell\ge2$ is an integer and $G$ is a graph on $n$ vertices and has average degree $d$, then  $\hom(C_{2\ell}, G) \ge d^{2\ell}$.
\end{lemma}
We now record a simple log-convexity property of the numbers of closed walks, which will be used to control the contribution of heavy vertices.
\begin{lemma}\label{lemma:bounding closed walks}
    Let $G$ be a graph.
    For all integers $\ell\ge2$ and $b\leq a\leq \ell/2$, and any vertex $v$, $$w_{vv}^{2a}w_{vv}^{2\ell-2a}\leq w_{vv}^{2b}w_{vv}^{2\ell-2b}.$$
\end{lemma}

\begin{proof}
    Let $A$ be the adjacency matrix of $G$. We know that the number of closed walks of length $k$ from $v$ to $v$ is $(A^{k})_{vv}$.
    So,
    \begin{align*}
    w_{vv}^{2k}=(A^{2k})_{vv}=e_{v}^TA^{2k}e_v=\langle A^{k+1}e_v,A^{k-1}e_v\rangle
    &\leq \lVert A^{k+1}e_v\rVert \lVert A^{k-1}e_v\rVert
    \\
    &=\sqrt{(A^{2k+2})_{vv} (A^{2k-2})_{vv}}=\sqrt{w_{vv}^{2k+2}w_{vv}^{2k-2}}.
    \end{align*}
    That is, $(w_{vv}^{2k})^2\leq w_{vv}^{2k+2}w_{vv}^{2k-2}.$ Therefore, $w_{vv}^k$ is log convex in terms of $k$ as wanted and so the result follows.
\end{proof}
With this, we can bound the proportion of heavy vertices at a given position on a path.

\begin{lemma}\label{lemma:rho_r small}
    Let $\ell \ge 2$ be an integer.
    Let $G$ be a $K$-almost regular graph on $n$ vertices with average degree $d>0$. Then
    $\rho_r\le K\left(\frac{t^{\ell-1}n}{d^{\ell}}\right)^{1/(\ell-1)}$, for every $1\leq r\leq \ell-1$.
\end{lemma}

\begin{proof}
By the definitions of $\rho_r$ and $\alpha_{xy}^r$, and \cref{lemma:H_2l eigenvalues and paths}, we have the following:
\begin{equation}\label{eq:rho_r times hom}
    \rho_r\hom(C_{2\ell}, G)=\sum_{(x,y)}(w_{xy}^\ell)^2\alpha_{xy}^r\leq t\sum_{(x,y)}\sum_{v}(p_{xy}^r(v))^2(w_{xy}^\ell)^2=t\sum_{(x,y)}\sum_v(w_{xv}^r)^2(w_{vy}^{\ell-r})^2,
\end{equation}
    where the inequality holds because $$\sum_v \left(p_{xy}^r(v)\right)^2\geq\sum_{p_{xy}^r(v)\geq \frac{1}{t}}\left(p_{xy}^r(v)\right)^2\ge\frac{1}{t}\sum_{p_{xy}^r(v)\geq \frac{1}{t}}p_{xy}^r(v)=\frac{\alpha_{xy}^r}{t}.$$
Now we observe that 
\begin{equation}\label{eq:double sum}
\sum_{(x,y)}\sum_v(w_{xv}^r)^2(w_{vy}^{\ell-r})^2=\sum_v\left(\sum_x(w_{xv}^r)^2\sum_y\left(w_{vy}^{\ell-r}\right)^2\right)=\sum_vw^{2r}_{vv}w_{vv}^{2(\ell-r)}.
\end{equation}
By \cref{lemma:bounding closed walks}, applied with $a = \min(r,\ell-r)$, $b = 1$, and $K$-almost regularity, we upper bound the right-hand side of \cref{eq:double sum}
\begin{equation}\label{eq:path product sum}
\sum_vw^{2r}_{vv}w_{vv}^{2\ell-2r}\leq\sum_vw_{vv}^2w_{vv}^{2\ell-2}=\sum_v\deg(v)w_{vv}^{2\ell-2}\leq Kd\sum_vw_{vv}^{2\ell-2}=Kd\sum_{i=1}^n\lambda_i^{2\ell-2},
\end{equation}
where the last equality is by \cref{lemma:H_2l eigenvalues and paths}.
From H\"older's inequality, we get
\begin{equation}\label{eq:Holder}
\left(\sum_{i=1}^n\lambda_i^{2\ell-2}\right)\leq\left(\sum_{i=1}^n\lambda_i^{2\ell}\right)^{\frac{\ell-2}{\ell-1}}\left(\sum_{i=1}^n\lambda_i^{2}\right)^{\frac{1}{\ell-1}}.
\end{equation}
Using the estimate $\sum_{i=1}^n\lambda_i^{2\ell} = \hom(C_{2\ell}, G) \ge d^{2\ell}$ (\Cref{lemma:many homs,lemma:H_2l eigenvalues and paths}), the observation $\sum_{i=1}^n \lambda_i^2 = nd$ alongside with \cref{eq:rho_r times hom,eq:double sum,eq:path product sum,eq:Holder}, we get 
$$\rho_r\leq\frac{t}{\hom(C_{2\ell}, G)}\sum_v w_{vv}^{2r}w_{vv}^{2\ell-2r}\leq Kdt\left(\frac{\sum_{i=1}^n\lambda_i^2}{\sum_{i=1}^n\lambda_i^{2\ell}}\right)^{\frac{1}{(\ell-1)}}\leq K\left(\frac{t^{\ell-1}n}{d^{\ell}}\right)^{\frac{1}{(\ell-1)}}.$$
\end{proof}

As mentioned before, our plan will be to exclude heavy vertices and control the remaining ones using the sparseness.
This will be done by the following lemma bounding the bilinear form $u^TA_{\Gamma}v$ on entrywise bounded subprobability vectors.

\begin{lemma}\label{lemma:bilinear form bounded}
    Suppose $u$ and $v$ are $n$-dimensional vectors with non-negative entries all at most $1/t$, and $1^Tu \le 1$, $1^Tv \le 1$, and $\Gamma$ is a $(1-\varepsilon,t)$-sparse graph on $n$ vertices with adjacency matrix $A_\Gamma$.
    Then $$u^TA_\Gamma v \le \varepsilon.$$
\end{lemma}

\begin{proof}
    Notice that the assumptions on $v$ translate to $0 \leq v_i \leq 1/t$ for all $i$ and $\sum_{i=1}^n v_i \leq 1$, where the $v_i$ are the components of $v$.
    Set $z = (z_1,z_2, \dots, z_n) := u^TA_{\Gamma}$ and note that $z \geq 0$ entrywise.
 With this, we obtain
    \begin{equation*}
        u^TA_{\Gamma}v = \sum_{i=1}^n z_iv_i.
    \end{equation*}
    Given the restrictions on the $v_i$ it is straightforward to see that the latter expression is maximized when $v$ is $1/t$ at the $t$ largest coordinates of $z$ and 0 elsewhere.
    Hence, it suffices to prove the statement for such $v$ which are indicators of size-$t$ sets of vertices scaled by $1/t$.
    Switching the roles of $u$ and $v$, we may also assume that $u$ is the indicator of a size-$t$ set of vertices scaled by $1/t$.
    Calling these sets $V$ and $W$, we get
    \begin{equation*}
        u^TA_{\Gamma}v = \frac{e(V,W)}{t^2} \leq \varepsilon,
    \end{equation*}
    using the $(1-\varepsilon,t)$-sparseness of $\Gamma$.
\end{proof}

Using this, and \cref{lemma:rho_r small}, we now bound the proportion of  homomorphisms having a particular chord in $\Gamma$.
We will then use the union bound (in \cref{lemma:holds in each part}) to bound the overall proportion of homomorphisms with any chord.

\begin{lemma}\label{lemma:small probability of chords}
    Let $\ell \ge 2$ be an integer.
  Let $\Gamma$ be a $(1-\varepsilon,t)$ sparse graph and $G\subseteq \Gamma$ be a $K$-almost regular subgraph with an average degree $d$ and $n$ vertices. For every fixed pair of nonconsecutive positions of $C_{2\ell}$, the proportion of  homomorphisms from $C_{2\ell}$ to $G$ which have an edge between them in $\Gamma$ is at most
  $$\varepsilon+2K\left(\frac{t^{\ell-1}n}{d^{\ell}}\right)^{\frac{1}{(\ell-1)}}.$$
\end{lemma}

\begin{proof}
Choose a  $\phi\in \operatorname{Hom}(C_{2\ell},G)$ uniformly at random.
By cyclically relabelling the cycle, suppose that the two positions are $0$ and $k$, where $2\leq k\leq \ell$.
Cut the cycle at positions $1$ and $\ell+1$.
We condition on the event $\phi(1)=x$ and $\phi(\ell+1)=y$.
The two halves of the cycle are then independent uniformly chosen length-$\ell$ walks from $x$ to $y$.
In the two halves, the prospective chord endpoints occur at internal positions $1$ and $(k-1)$.
For an internal position $r$, the corresponding conditional distribution over vertices appearing at this position is $p_{xy}^r(v)$. We discard heavy vertices at position $r$.
Then the total mass of the discarded vertices is $\alpha_{xy}^r$.

We define a subprobability vector for an internal position $r$ and vertices in $\Gamma$, denoted $\mathbf{u}_r$,
where $$(\mathbf{u}_r)_v = \begin{cases}
    0 & \text{if $v$ was discarded or $v\notin V(G)$}
    \\ p_{xy}^r(v) & \text{otherwise.}
 \end{cases}$$
Note that in $\mathbf{u}_r$, each entry is less than $1/t$, and the total mass is at most one.
 Then, the probability of the $0$-th and $k$-th vertices in the cycle being connected and not discarded is
$\mathbf{u}_1^TA_\Gamma\mathbf{u}_{k-1}$, which, by \cref{lemma:bilinear form bounded} applied to $\mathbf{u}_1$ and $\mathbf{u}_{k-1}$ is at most $\varepsilon$.
Thus, in total 
$$\mathbb{P}(\phi(0)\phi(k)\in E(\Gamma)| \phi(1)=x  ,\phi(\ell+1)=y)\leq\varepsilon+\alpha_{xy}^1+\alpha_{xy}^{k-1}.$$
Averaging over the  probability $\frac{\left(w_{xy}^\ell\right)^2}{\hom(C_{2\ell}, G)}$  of $(x,y)$ being chosen gives 
$$\mathbb{P}(\phi(0)\phi(k)\in E(\Gamma))\leq \varepsilon+\rho_1+\rho_{k-1}.$$
Applying \cref{lemma:rho_r small} gives the desired bound.
\end{proof}
\section{Technical Lemmas and Proof of the Main Theorem}
\label{section:technical}

In this section, we state two technical lemmas, which we will use, together with \cref{lemma:small probability of chords}, to prove \cref{theorem:main theorem}.
First, Janzer \cite{JanzerOliver2023RTno} provides a lemma which we use to prove a bound on non-injective homomorphisms:

\begin{lemma}[Lemma 2.2, \cite{JanzerOliver2023RTno}]
\label{lemma:many injective homs0}
Let $\ell\geq2$ be an integer and let $G=(V,E)$ be a graph on $n$ vertices. Let $\sim$ be a symmetric binary relation defined over $V$ such that for every $u\in V$ and $v\in V$, $v$ has at most $s$ neighbours $w\in V$ which satisfy $u\sim w$. Then the number of homomorphic $2\ell$-cycles $(x_1,x_2,\dots,x_{2\ell})$ in $G$ such that $x_i\sim x_j$ for some $i\neq j$ is at most
\[
32\ell^{\frac{3}{2}}s^{\frac{1}{2}}\Delta(G)^{\frac{1}{2}}n^{\frac{1}{2\ell}}
\hom(C_{2\ell},G)^{1-\frac{1}{2\ell}}.
\]
\end{lemma}

We can already use this to prove the result for a $K$-almost regular subgraph $G\subseteq\Gamma$ of high average degree.

\begin{lemma}\label{lemma:holds in each part}
For every integer $\ell\geq2$ and every $K\geq1$, there exist constants
$0<\varepsilon<1$ and $C_1,C'_1>0$ with the following property.
Let $\Gamma$ be a $(1-\varepsilon,t)$-sparse graph, and let $G\subseteq\Gamma$ be a $K$-almost-regular subgraph on $m$ vertices with
average degree $d$. If $G$ has at least
$ C_1t^{1-1/\ell}m^{1+1/\ell}$ edges,
then $G$ contains at least $C'_1d^{2\ell}$ copies of $C_{2\ell}$ that
are induced in $\Gamma$.
\end{lemma}

\begin{proof}
Fix $\ell$ and $K$.
Given a homomorphism $\phi\in\operatorname{Hom}(C_{2\ell},G)$, it can fail to be an induced cycle in $\Gamma$ in two ways:
\begin{enumerate}
    \item $\phi$ is not injective, or
    \item if its image is $x_1x_2\dots x_{2\ell}x_1$, there is an edge $x_ix_j\in E(\Gamma)$ for $i\not\equiv j\pm1\pmod{2\ell}$.
\end{enumerate}
By \cref{lemma:many homs}, $\hom(C_{2\ell},G)\geq d^{2\ell}$. Let $\sim$ be equality and take $s=1$ in \cref{lemma:many injective homs0}. Since $\Delta(G)\leq Kd$, the number of homomorphisms that are not injective is at most
\begin{align*}
&32\ell^{\frac{3}{2}}\Delta(G)^{\frac{1}{2}}m^{\frac{1}{2\ell}}
\hom(C_{2\ell},G)^{1-\frac{1}{2\ell}}\\
&\quad\leq
32\ell^{\frac{3}{2}}K^{\frac{1}{2}}d^{\frac{1}{2}}m^{\frac{1}{2\ell}}
\hom(C_{2\ell},G)^{1-\frac{1}{2\ell}}\\
&\quad\leq
32\ell^{\frac{3}{2}}K^{\frac{1}{2}}d^{-\frac{1}{2}}m^{\frac{1}{2\ell}}
\hom(C_{2\ell},G)\\
&\quad\leq
32\ell^{\frac{3}{2}}K^{\frac{1}{2}}(2C_1)^{-\frac{1}{2}}
t^{-\frac12+\frac{1}{2\ell}}
\hom(C_{2\ell},G).
\end{align*}
Since $t\geq1$, by choosing $C_1$ large enough, this is at most $\frac14\hom(C_{2\ell},G)$.
By \cref{lemma:small probability of chords}, for each pair of non-adjacent positions on the cycle, at most
\[
\left(\varepsilon+
2K\left(\frac{t^{\ell-1}m}{d^\ell}\right)^{\frac{1}{\ell-1}}\right)\hom(C_{2\ell},G)
\]
 homomorphisms have an edge in $\Gamma$ between those positions. Hence, the proportion of  homomorphisms that have a chord in $\Gamma$ is at most
\begin{align*}
(2\ell)^2\left[
\varepsilon+
2K\left(\frac{t^{\ell-1}m}{d^\ell}\right)^{\frac{1}{\ell-1}}
\right]
&\leq
(2\ell)^2\left[
\varepsilon+
2K(2C_1)^{-\frac{\ell}{\ell-1}}
\right] \leq\frac14
\end{align*}
if we choose, for example, $\varepsilon=\frac{1}{32\ell^2}$ and $C_1$ to be large enough. Hence by combining the two estimates, there are at least $\frac12\hom(C_{2\ell},G)$ injective homomorphisms from $C_{2\ell}$ to $G$ that are induced in $\Gamma$. Each copy of $C_{2\ell}$ is counted $4\ell$ times, twice per vertex, once in each direction. Therefore, there are at least
\[
\frac{1}{8\ell}\hom(C_{2\ell},G)
\geq\frac{1}{8\ell}d^{2\ell}
\]
copies of $C_{2\ell}$ in $G$ which are induced in $\Gamma$. Thus, we may take $C'_1=1/(8\ell)$.
\end{proof}

The second lemma, due to Jiang and Yepremyan, allows us to get the $K$-almost regularity, which was used in \cref{section:inducedness} and in the preceding lemma to upper bound the maximum degree.

\begin{lemma}[Lemma 3.2, \cite{jiangYepremyan2020}]
\label{lemma:regularization_citation}
Let $\alpha$ be real and $s,q$ integers where $0<\alpha<1$ and $q\geq s\geq1$. Then there exist positive reals $C_0=C_0(\alpha,s,q)$ and $K=K(\alpha,s,q)$ such that the following holds. For every $C\geq C_0$, if $G$ is an $n$-vertex graph with $e(G)\geq Cn^{1+\alpha}$, then $G$ contains a collection of edge-disjoint subgraphs $G_1,\dots,G_m$ satisfying:
\begin{itemize}
    \item for every $i\in[m]$, $G_i$ is $K$-almost-regular and satisfies
$e(G_i)\geq\frac14Cv(G_i)^{1+\alpha}$
    \item $\sum_{i=1}^m f(G_i,s,q)\geq\frac{1}{4^s}f(G,s,q),$
\end{itemize}
where $f(H,s,q)=e(H)^s/v(H)^q$.
\end{lemma}

We 
now have everything to prove the main result.

\begin{proof}[Proof of \cref{theorem:main theorem}]
Fix $\ell\geq2$.
Let $C_0$ and $K$ be the constants from \cref{lemma:regularization_citation} for $\alpha=1/\ell$ and $s=q=2\ell$. Let $\varepsilon$, $C_1$ and $C'_1$ be the constants from \cref{lemma:holds in each part} for this value of $K$, and take $C=\max(4C_1,C_0)$ and
$C'=\frac{C'_1}{4^{2\ell}}.$ Suppose $\Gamma$ is $(1-\varepsilon,t)$-sparse for some $t$, and has at least $Ct^{1-1/\ell}n^{1+1/\ell}$ edges. Set $D=Ct^{1-1/\ell}$. Since $t\geq1$ and $C\geq C_0$, we have $D\geq C_0$, so we may apply \cref{lemma:regularization_citation} to $\Gamma$. We obtain edge-disjoint subgraphs $G_1,\dots,G_m$. Writing $n_i=v(G_i)$ and $d_i=d(G_i)$, each $G_i$ is $K$-almost regular and
$$
e(G_i)\geq\frac{D}{4}n_i^{1+\frac{1}{\ell}}
=\frac{C}{4}t^{1-\frac{1}{\ell}}n_i^{1+\frac{1}{\ell}}
\geq C_1t^{1-\frac{1}{\ell}}n_i^{1+\frac{1}{\ell}}.
$$
Thus, we may apply \cref{lemma:holds in each part} to each $G_i$ to see that each $G_i$ contains at least $C'_1d_i^{2\ell}$ copies of $C_{2\ell}$ that are induced in $\Gamma$. Moreover, since $f(H,2\ell,2\ell)=2^{-2\ell}d(H)^{2\ell}$, the second conclusion of \cref{lemma:regularization_citation} gives
$$
\sum_{i=1}^m d_i^{2\ell}
\geq\frac{1}{4^{2\ell}}d(\Gamma)^{2\ell}.$$

As the $G_i$ are edge-disjoint, we do not double count any cycle, and so $\Gamma$ contains at least
$$
\sum_{i=1}^m C'_1d_i^{2\ell}
\geq\frac{C'_1}{4^{2\ell}}d(\Gamma)^{2\ell}
=C'd(\Gamma)^{2\ell}
\geq C'2^{2\ell}C^{2\ell}t^{2\ell-2}n^2
$$
induced copies of $C_{2\ell}$, as required.
\end{proof}
\section{Concluding Remarks}\label{section:conclusion}

In this paper, we proved \cref{theorem:main theorem}, which partially resolved \cref{problem:the problem} in the "bootstrapped" case of sufficiently locally sparse graphs.

The obvious lines for further research are to come up with a proof for general $(c,t)$-sparse graphs with sufficient edge density or to find a counterexample for some small $c$.
On the one hand, the existence result of Ding et al. \cite{ding2024induced} for any $(c,t)$-sparseness condition suggests that the supersaturation result might also be true in that case.
On the other hand, our proof uses $\varepsilon = O\left(\ell^{-2}\right)$ in a  canonical way, which may be seen as evidence for the contrary.
If there were indeed a counterexample for a small $c$, this would mean there is a threshold where we start getting enough induced cycles, which would be an interesting avenue for future study.

\subsection*{Acknowledgments} This work was started at the Young Researchers in Mathematics Program at the Bernoulli Center at EPFL in July 2026, supervised by Oliver Janzer with further input from Rik Sarkar.
We thank them as well as the other participants and the other organisers Florian Richter, Mats Stensrud and Sven Wang for the helpful and friendly environment.

\begingroup
\setlength{\bibsep}{0pt} 
\small 
\bibliographystyle{plainnat}
\bibliography{bibliography}
\endgroup

\end{document}